\documentclass[12pt]{article}
\usepackage{amsmath,amsopn,amssymb,amsthm,amsfonts}
\usepackage[T2A]{fontenc}

\newtheorem{theorem}{Theorem}[section]
\newtheorem{corollary}[theorem]{Corollary}

\newtheorem{definition}[theorem]{Definition}
\newtheorem{lemma}[theorem]{Lemma}

\def\dim {{\rm dim}}
\def\div {{\rm div}}

\def\grad {{\rm grad}}

\begin{document}

\vspace{7mm}

\centerline{\large \bf Bi-invariant metric on symplectic diffeomorphisms  group}

\vspace{3mm}

\centerline{\large \bf N.~K.~Smolentsev }

\vspace{7mm}

\begin{abstract}
We show the existence of a weak bi-invariant symmetric nondegenerate 2-form  on  the symplectic  diffeomorphisms  group $\mathcal{D}_\omega$ of a symplectic Riemannian manifold $(M,g,\omega)$ and study its properties. We describe the Euler's equation on a Lie algebra of group $\mathcal{D}_\omega$ and calculate the sectional curvature of $\mathcal{D}_\omega (T^n)$.
\end{abstract}

\section{Preface}
Let $M$ be a compact manifold and let $G=Diff(M)$ be the group of all smooth (of class $C^\infty$) diffeomorphisms on $M$, with group operation being composition. The group $Diff(M)$ is a infinite-dimensional Frech\'{e}t manifold and there are nontrivial problems with the notion of smooth maps between Frech\'{e}t spaces. There is no canonical extension of the differential calculus from Banach spaces (which is the same as for $\mathbb{R}^n$) to Frech\'{e}t spaces.
It is possible to use the completion $Diff(M)$ in the Banach $C^k$-norm, $0\le  k< \infty$, or in
the Sobolev $H^s$-norm, $s>\dim(M)/2$. Then $Diff^k(M)$ and $Diff^s(M)$
become Banach and Hilbert manifolds, respectively. Then we consider the inverse limits of these Banach and Hilbert manifolds, respectively: $Diff(M)= \lim_{\leftarrow} Diff^k(M)$ becomes a so-called ILB- (Inverse Limit of Banach) Lie group, or with the Sobolev topologies
$Diff(M)= \lim_{\leftarrow} Diff^s(M)$ becomes a so-called ILH- (Inverse Limit of Hilbert) Lie group.
The main theorems of differential calculus are true for (Banach) Hilbert manifolds. The results are proved for $Diff^s(M)$ and are then extended to the Lie-Frech\'{e}t group $Diff(M)$.
See the study by Omori in \cite{Omo5} for details.
In \cite{Eb-Mar}, Ebin and Marsden showed that the group $\mathcal{D}_\omega = \{\eta\in Diff(M);\, \eta^*(\omega) =\omega \}$ of smooth diffeomorphisms preserving a symplectic 2-form $\omega$ on $M$ is ILH-Lie group.

Let's remind concept of ILH-Lie group.
A topological vector space $\mathbb{E}$ is called an \emph{ILH-space} if $\mathbb{E}$ is the inverse limit of the Hilbert spaces $\{E^s\}$ enumerated  by  integers $s\geq d \geq 0$, and, moreover, $E^{s+1}$ is  linearly  and  densely  embedded in $E^s$.
Denote
\[
\mathbb{E}=\lim_{\leftarrow}E^s= \cap_{s\geq d}E^s .
\]
A system of vector spaces $\{\mathbb{E},E^s,  s\in \mathbb{N}(d)\}$, where $\mathbb{N}(d)=\{s\in\mathbb{N};\ s\geq d \geq 0 \}$, is also called a \emph{Sobolev chain}.

\begin{definition} [\cite{Omo5}] \label{Def2.5}
A topological group $G$ is called a strongly ILH-Lie group modeled on a chain
$\{\mathbb{E},E^s,  s\in \mathbb{N}(d)\}$
if  there  exists  a  system $\{G^s, s\in \mathbb{N}(d)\}$ of  topological  groups $G^s$ satisfying  the following conditions:

 {\rm (G1)}  every group $G^s$ is a smooth Hilbert manifold modeled on $E^s$;

 {\rm (G2)} $G^{s+1}$ is a dense subgroup in $G^s$, and the embedding $G^{s+1}\subset G^s$ is a mapping of class $C^{\infty}$;

 {\rm (G3)} $G=\cap\, G^s$ with the inverse limit topology;

 {\rm (G4)}  the  group  multiplication $G\times G \rightarrow G$, $(\eta,\zeta)\rightarrow \eta\zeta$ extends  to  a  mapping $G^{s+l}\times G^{s}\rightarrow G^{s}$ of class $C^l$;

 {\rm (G5)} the mapping $G\rightarrow G$, $\eta \rightarrow \eta^{-1}$  extends to a mapping $G^{s+l}\rightarrow G^{s}$ of class $C^l$;

 {\rm (G6)}  for each $\eta\in G^{s}$, the right translation $R_\eta:G^{s}\rightarrow G^{s}$ is a mapping of class $C^\infty$;

 {\rm (G7)}  let $T_e G^s$  be  the  tangent  space  of $G^s$ at  the  identity $e\in G^s$, and let $TG^s$ be  the  tangent  bundle. The mapping $dR: T_e G^{s+l}\times G^s \rightarrow TG^s$ defined by $dR(u,\eta)=dR_\eta u$ is a mapping of class $C^l$;

 {\rm (G8)}  there  exist  an  open  neighborhood $U$ of  zero  in $T_e G^d$ and  a $C^{\infty}$-diffeomorphism of $U$ onto  an  open    neighborhood $\widetilde{U}$ of  the  unity $e\in G^d$, $\varphi(0)=e$, such  that the  restriction  of $\varphi$ to $U\cap T_e G^s$  is  a $C^{\infty}$-diffeomorphism  of  the  open  subset $U\cap T_e G^s$ from $T_e G^s$ onto  an  open  subset $\widetilde{U}\cap G^s$ from $G^s$ for  any $s\geq d$.
\end{definition}

\textbf{Remark 1.}
Roughly speaking, condition {\rm (G8)} means that a coordinate neighborhood of the identity in each of the groups $G^s$ can be chosen independently of $s$.  Then, setting $T_e G=\cap\, T_e G^s$ with the inverse limit topology, we see that $\varphi:U\cap T_e G  \rightarrow  \widetilde{U}\cap G$ is a homeomorphism defining the Lie--Frech\'{e}t group structure on $G$.
\vspace{1mm}

The pair $(U,\varphi)$ in condition {\rm (G8)} is called the \emph{ILH-coordinates}  on $G$ in a neighborhood of the unity.

\begin{definition}  \label{Def2.7}
A topological group $G$ is called an ILH-Lie group if there exists a system of topological groups $\{G^s, s\in \mathbb{N}(d)\}$ satisfying conditions {\rm (G1)}--{\rm(G7)}.
\end{definition}

Omori  showed  in \cite{Omo1}, \cite{Omo5} that  the   group $\mathcal{D}$ of  smooth diffeomorphisms  of  a  compact  manifold $M$ is  a  strongly  ILH-Lie  group  modeled  on the space $\{\Gamma(TM),\Gamma^s(TM); s\in \mathbb{N}(\dim M+5)\}$ of smooth vector fields on $M$, where $\Gamma^s(TM)$ is the space of vector fields of Sobolev class of smoothness $H^s$.

In \cite{Eb-Mar}, Ebin and Marsden showed that the following groups are ILH-Lie groups:

   (1)  the group $\mathcal{D}_\mu$ of smooth diffeomorphisms preserving a volume element $\mu$ on the manifolds $M$;

   (2)  the group $\mathcal{D}_\omega$ of smooth diffeomorphisms preserving a symplectic structure $\omega$ on $M$.

\section{Symplectic diffeomorphisms group and bi-invariant metric} \label{Sympl-Diff-Gr}
Let $M$ be a smooth (of class $C^\infty$) compact orientable manifold of dimension $2n$ without boundary. The manifold $M$ is said to be \emph{symplectic} if a closed nondegenerate 2-form $\omega$ is given on it.
In this case, there exists an almost complex structure $J$ on $M$ having  the  following  properties:  $\omega(JX,JY)=\omega(X,Y)$, \ and $\omega(X,JX)>0$  for  any $X,Y \in \Gamma(TM)$.  The formula
\begin{equation}
g(X,Y)=\omega(X,JY), \qquad   X,Y \in \Gamma(TM), 
\label{g(X,Y)-10.1}
\end{equation}
defines an almost Hermitian structure $(J,g)$ on $M$ whose fundamental form is $\omega$. Therefore, $(M,\omega, J, g)$ is an almost K\"{a}hler structure on $M$.  Note that $\omega^n = n!\mu$, where $\mu$ is the Riemannian volume element on $M$.
The form $\omega$ defines the bundle isomorphism $\iota:TM \rightarrow T^*M$, $\iota(V)=-\iota_V \omega =\omega( ., V)$.

A transformation $\eta:M\rightarrow M$ is  said  to  be  \emph{symplectic}  if  it  preserves  the  symplectic  form $\omega$, i.e.,  if $\eta^*\omega = \omega$. Let $\mathcal{D}_\omega$ be  the  group  of  all  smooth  symplectic  diffeomorphisms  of  a  manifold $M$. Ebin and Marsden showed in \cite{Eb-Mar} that the group $\mathcal{D}_\omega$ is a closed ILH-Lie subgroup of the diffeomorphisms group $\mathcal{D}$.
Omori proved in \cite{Omo5} that $\mathcal{D}_\omega$ is a closed strongly ILH-Lie subgroup of the group $\mathcal{D}$.

The Lie algebra of the group $\mathcal{D}_\omega$ consists of all vector fields that infinitesimally preserve the form $\omega$, i.e., those vector fields $X$ on $M$ for which $L_X \omega = 0$, where $L_X = i_X \circ d + d \circ i_X$ is the Lie derivative.  Such vector fields are said to be \emph{locally Hamiltonian}.  Their characteristic property is that the form $i_X \omega =\omega(X,.)$ is closed.  Indeed,
$$
L_X w = i_X (d\omega) + d(i_X\omega)= d(i_X\omega)=0.
$$

A vector field $X$ on $M$ is said to be \emph{Hamiltonian} if the form $-i_X \omega$ is exact, i.e., when it is the differential of a  certain  function $F$ on $M$: $\omega(.,X )=dF$. The  function $F$     is  called  the  \emph{Hamiltonian function} of  the  field $X$, and in this case, the vector field $X$  is denoted by $X_F$. It well-known that $X_H=J \, \grad H$.

The  set  of  Hamiltonian  vector  fields  on $M$ form a Lie  algebra  with  respect  to  the  Lie  bracket of  vector  fields;  moreover, $[X_F,X_H]=X_{\{F,H\}}$,  where $\{F,H\}=\omega(X_F,X_H)$ is  the  Poisson  bracket  of functions $F$ and $H$ on the symplectic manifold $M$.

Denote by $\Gamma_\omega(TM)$ the Lie algebra of smooth locally Hamiltonian vector fields on $M$.  Let
$$
\Gamma_{\omega \partial}(TM)=\{X\in \Gamma_\omega(TM);\ i_X\omega - \mbox { is an exact form } \}.
$$
It  is  easy  to  see  that $\Gamma_{\omega \partial}(TM)$ is  an  ideal  of $\Gamma_\omega(TM)$, since $i_{[X,Y]}\omega=d(i_X(i_Y\omega))$ for $X,Y\in \Gamma_\omega(TM)$.
Since $\Gamma_\omega(TM)/\Gamma_{\omega \partial}(TM)=H^{1}(M)$, it follows  that $\Gamma_{\omega \partial}(TM)$ is of finite codimension in $\Gamma_\omega(TM)$. The ILH-Lie group corresponds to the ideal $\Gamma_{\omega \partial}(TM)$.

\begin{theorem} [see \cite{Omo5}, Theorem 8.5.1] \label{Th3.10}
There exists a strong ILH-subgroup $\mathcal{D}_{\omega\partial}$ of the group $\mathcal{D}_{\omega}$ such that $\Gamma_{\omega \partial}(TM)$ is its Lie algebra.
\end{theorem}

The connected component of the identity of the group $\mathcal{D}_{\omega\partial}$ is called the  \emph{Hamiltonian  transformation} group of the manifold $M$ and is denoted by $\mathcal{D}_{\omega H}$ in what follows. The Lie algebra of the group $\mathcal{D}_{\omega H}$  is $\Gamma_{\omega \partial}(TM)$.

\vspace{1mm}
\textbf{Remark 2.}
Ebin and Marsden  showed  in  \cite{Eb-Mar} that  the  connected  component $\mathcal{D}_{0}$ of  the  diffeomorphism group of the manifold $M$ is diffeomorphic to $\mathcal{D}_{\mu}\times \mathcal{V}$, the direct product of the group $\mathcal{D}_{\mu}$ of volume-preserving
diffeomorphisms and the convex space $\mathcal{V}$ of volume elements. In the symplectic case, this fact does not  hold. McDuff presented examples \cite{McD3} of cases where the group $\mathcal{D}_{\omega}$ is not homotopy equivalent to the group $\mathcal{D}_{0}$.
\vspace{1mm}

{\bf Bi-invariant metric on $\mathcal{D}_\omega$}\\
Let $X,Y\in T_e\mathcal{D}_{\omega}= \Gamma_\omega(TM)$.
The right-invariant weak Riemannian structure on the group $\mathcal{D}_\omega$ is defined by
\begin{equation}
(X,Y)_e = \int_M g(X,Y)d\mu.  
\label{Weak-Riem-D-omega-10.2}
\end{equation}

Let $X=X_H$ be a Hamiltonian vector field on $M$.
For the Hamiltonian $H$ to be uniquely defined by its vector field $X$, we assume that
$$
\int_M H(x)d\mu(x)= 0.
$$
Denote  by $C_0^\infty(M,\mathbb{R})$ the  space of  such  functions  on $M$. The  Laplacian  $\Delta=- \div \circ \grad$ defines  the isomorphism $\Delta: C_0^\infty(M,\mathbb{R}) \rightarrow C_0^\infty(M,\mathbb{R})$. Then  the  operator $\Delta^{-1}$ inverse  to  the  Laplacian  is  defined.
Let $H^s_0(M)$ be the completion of the space $C_0^\infty(M,\mathbb{R})$ with respect to the $H^s$-norm of the Sobolev space, $s\geq  2n+5$. The operator $\Delta$ extends to a Hilbert space isomorphism $\Delta:H^s_0(M) \rightarrow H^{s-2}_0(M)$.

Consider  the  Lie  algebra $\Gamma_{\omega \partial}(TM)$ of  smooth  Hamiltonian  vector  fields  on $M$. In \cite{Omo5} and \cite{Rat-Shm},  it  was  shown that  there  exists  a  connected  ILH-Lie  group $\mathcal{D}_{\omega H}$ whose  Lie  algebra is $\Gamma_{\omega \partial}(TM)$. If  the  class $[\omega]\in H^2(M,\mathbb{R})$ is integral,  then  the  group $\mathcal{D}_{\omega H}$ is  (see \cite{Rat-Shm})  the  commutator $[\mathcal{D}_{\omega 0},\mathcal{D}_{\omega 0}]$ for  the  connected  component  $\mathcal{D}_{\omega 0}$ of  the  identity  of  the  group  $\mathcal{D}_\omega$.  If  the  first  cohomology  group  is  trivial, $H^1(M,\mathbb{R})=0$,  then $\mathcal{D}_{\omega H} = \mathcal{D}_{\omega 0}$. 

Define the inner product on the algebra $\Gamma_{\omega \partial}(TM)$ by
\begin{equation}
\left<X_F,X_H\right>_e=\int_M F(x)H(x)d\mu(x).    
\label{(XF,XH)-D-omega-10.3}
\end{equation}
It is easy to see that this inner product is bi-invariant, or, in other words, it is invariant with respect to the adjoint action of the group $\mathcal{D}_{\omega}$ on the algebra $\Gamma_{\omega \partial}(TM)$.  Indeed, for $\eta \in \mathcal{D}_\omega$, we have $Ad_\eta
X_F = dL_\eta dR_\eta^{-1} X_F = X_{F\circ \eta^{-1}}$ and $\eta^*\muь=\eta^*(c\omega^n)= c(\eta^*\omega)^n=c\omega^n=\mu$, where $c=1/{n!}$. The infinitesimal variant of the invariance of the inner product (\ref{(XF,XH)-D-omega-10.3}) has the form
\begin{equation}
\left<[X_G,X_F],X_H\right>_e+\left<X_F,[X_G,X_H]\right>_e =0, 
\label{(Eq-10.4}
\end{equation}
where $[X_G,X_F]$ is the Lie bracket of Hamiltonian vector fields $X_G$ and $X_F$ on $M$.  Recall that $[X_G,X_F]=X_{\{G,F\}}$, where $\{G,F\}$ is  the  Poisson  bracket  of  functions  $G$ and $F$ on  a  symplectic  manifold.
The  bi-invariant metric (\ref{(XF,XH)-D-omega-10.3}) was introduced in \cite{Smo12}, and it was also considered in \cite{Hof}, \cite{El-Ra1} and \cite{LaM}.

\begin{theorem} [see \cite{Smo12}] \label{Th10.1}
The inner product (\ref{Weak-Riem-D-omega-10.2}) is expressed through the bi-invariant inner product (\ref{(XF,XH)-D-omega-10.3}) on $\Gamma_{\omega \partial}(TM)$ as follows:
\begin{equation}
(X_F,X_H)_e= \left<X_{\Delta F}, X_H\right>_e, 
\label{(Eq-10.5}
\end{equation}
where $\Delta= -\div\circ\grad$ is the Laplace operator.
\end{theorem}

\begin{proof}
Indeed, we have
$$
(X_F,X_H)_e=\int_M g(X_F,Y_H) d\mu =\int_M g(J\, \grad F, J\, \grad H) d\mu =
$$
$$
= \int_M g(\grad F, \grad H) d\mu =\int_M -\div (\grad F)\, H\, d\mu =\int_M (\Delta F) H\, d\mu.
$$
\end{proof}

We can extend inner product (\ref{(XF,XH)-D-omega-10.3}) defined on the tangent space $T_e \mathcal{D}_{\omega H}=\Gamma_{\omega \partial}(TM)$ at the identity to the whole group $\mathcal{D}_{\omega H}$ using right translations.  The invariance of form (\ref{(XF,XH)-D-omega-10.3}) implies the bi-invariance of the weak Riemannian structure on $\mathcal{D}_{\omega H}$ being obtained.

For the group $\mathcal{D}_{\omega H}$, there exists (see \cite{Omo5}) a system $\{\mathcal{D}_{\omega H}^s;\ s\geq 2n+5\}$ of topological groups $\mathcal{D}_{\omega H}^s$ satisfying properties (G1) -- (G7) above. Each  of  the  topological  groups $\mathcal{D}_{\omega H}^s$ is a smooth  Hilbert  manifold modeled  on  the  space $\Gamma_{\omega \partial}(TM)^s$ of  Hamiltonian  vector  fields of Sobolev  class  of  smoothness $H^s$.
Therefore the tangent space $T_e \mathcal{D}_{\omega H}^s$ at the identity is identified with $\Gamma_{\omega \partial}(TM)^s$. According to property (G6) of an ILH-Lie group,  the  right  translation $R_\eta :\mathcal{D}_{\omega H}^s \rightarrow \mathcal{D}_{\omega H}^s$ is a smooth  mapping  for  any $\eta \in \mathcal{D}_{\omega H}^s$. This  allows  us  to  obtain  the right-invariant weak Riemannian structure on $\mathcal{D}_{\omega H}^s$ from the inner product (\ref{(XF,XH)-D-omega-10.3}) on $\Gamma_{\omega \partial}(TM)^s =T_e\mathcal{D}_{\omega H}^s$.

\begin{theorem} [see \cite{Smo12}]  \label{Th10.2}
The weak Riemannian structure on the smooth Hilbert manifold $\mathcal{D}_{\omega H}^s$ obtained
from  the  inner  product (\ref{(XF,XH)-D-omega-10.3}) on $T_e\mathcal{D}_{\omega H}^s$ by using right  translations  on $\mathcal{D}_{\omega H}^s$ is  smooth. The  corresponding  weak Riemannian structure on the group $\mathcal{D}_{\omega H}=\cap\, \mathcal{D}_{\omega H}^s$ is bi-invariant and ILH-smooth.
\end{theorem}

\section{Euler equation}\label{Euler-eq-10.1}

On the algebra $\Gamma_{\omega \partial}(TM)$ of Hamiltonian vector fields, there exist the bi-invariant inner product (\ref{(XF,XH)-D-omega-10.3}) and the kinetic energy function $L$,
\begin{equation}
L(X_F)= \frac 12(X_F,X_F)_e=\frac 12\int_M g(X_F,X_F) d\mu,\quad
X_F\in \Gamma_{\omega \partial}(TM). 
\label{(L(XF)-Eq-10.6}
\end{equation}
The function $L$ is written as follows through the invariant inner product (\ref{(XF,XH)-D-omega-10.3}):
$$
L(X_F)= \frac 12 \left<X_F,X_{\Delta F}\right>_e.
$$
The Legendre transform \cite{Abr-Mar} corresponding to the function $L$ has the form $Y_F = X_{\Delta F}$. Then the Hamiltonian function has the expression $H(Y_F)= L(X_F)=\frac
12\left<Y_{\Delta^{-1}F},Y_F\right>_e$. It is easy to calculate the gradient of the function $H(Y_F)$ with respect to the invariant inner product (\ref{(XF,XH)-D-omega-10.3}):
$$
\grad H(Y_F)=Y_{\Delta^{-1}F}.
$$
Recall that the operator $\Delta:\Gamma_{\omega \partial}(TM) \rightarrow \Gamma_{\omega \partial}(TM)$ is an isomorphism, and, therefore, it has the inverse operator $\Delta^{-1}$.

Following the general construction proposed in \cite{Mi-Fo}, we write the Euler equation
\begin{equation}
\frac {d}{dt}Y_F= [Y_F, \grad H(Y_F)]=[Y_F,Y_{\Delta^{-1}F}]. 
\label{(Euler-Eq-10.7}
\end{equation}
Passing from the vector fields to their Hamiltonians, we arrive at the Euler equation of the form $\frac {\partial F}{\partial t}= \{F,\Delta^{-1}F\}$, where $\{ , \}$ is the Poisson bracket on the symplectic manifold $M$.  Make the change $F :=\Delta^{-1}F$ then
\begin{equation}
\frac {\partial }{\partial t}\Delta F= \{\Delta F,F\}. 
\label{(Euler-Eq-10.8}
\end{equation}
We consider the Euler equation precisely in this form.

\vspace{1mm}
\textbf{Remark 3.}  For $n=1$, the  Euler  equation  coincides  with  the  Helmholtz  equation  of  motion  of  the two-dimensional ideal incompressible fluid.
\vspace{1mm}

\begin{theorem} [see \cite{Smo12}] \label{Th10.3}
For any function $F_0(x)\in H_0^s(M)$, $s\geq 2n+7$, there exists a unique continuous  solution $F(x,t)$ of  Eq. (\ref{(Euler-Eq-10.8}) defined  on $(-\varepsilon,\varepsilon)\times M$ for a certain $\varepsilon > 0$ and  having  the  following properties :

{\rm (1)} \ $F(0,x)=F_0(x)$;

{\rm (2)} if $F_0(x)\in H_0^{s+l}(M),\ l\geq 0$, then $F(t,x)\in H_0^{s+l}(M)$ for any $t\in (-\varepsilon,\varepsilon)$;

{\rm (3)}  the  flow $\eta_t$ on $M$ generated  by  the  Hamiltonian  vector  field $X_{F(t,x)}$ is a geodesic on $\mathcal{D}_{\omega H}^s$ of  the right-invariant metric (\ref{Weak-Riem-D-omega-10.2}). Conversely, if $\eta_t$ is a geodesic on $\mathcal{D}_{\omega H}^s$, then the velocity field
$$
X_F= dR_{\eta_t}^{-1}(\frac {d}{dt}\eta_t)
$$
has the Hamiltonian $F$ satisfying Eq. (\ref{(Euler-Eq-10.8}).
\end{theorem}

We write the Euler equation $\frac {\partial }{\partial t}\Delta F = \{\Delta F,F\}$ through vector fields in the form
$$
\frac {\partial }{\partial t}X_{\Delta F}= \left[X_{\Delta F},X_F\right]= -L_{X_F}X_{\Delta F}.
$$
Therefore, the vector field $X_{\Delta F}$ on $M$ is transported by the flow $\eta_t$ of the vector field $X_F$. In other words, $X_{\Delta F}(t,x)= dL_{\eta_t} dR_{\eta_t}^{-1} (X_{\Delta F})$. Since $Ad_{\eta_t} X_{H(t)}=X_{H(\eta_t^{-1}(x))}$, the  Hamiltonian $\Delta F(t,x)$ is  transported by the flow $\eta_t$:
$$
\Delta F(t,x)=(\Delta F)(\eta_t^{-1}(x)).
$$
This immediately implies the following theorem.

\begin{theorem} [see \cite{Smo12}] \label{Th10.4}
Let $F= F(t,x)$ be a solution of the Euler equation (\ref{(Euler-Eq-10.8}), and let $\eta_t$ be the flow on $M$ generated by the vector field $X_F$. Then the following quantities are independent of time $t$:
\begin{equation}
L=\frac 12\left(X_F,X_F \right)_e=\frac 12\int_M F\Delta F\ d\mu, 
\label{(L-Eq-10.9}
\end{equation}
\begin{equation}
I_k=\int_M \left(\Delta F\right)^k\ d\mu. 
\label{(Ik-Eq-10.10}
\end{equation}
\end{theorem}

\textbf{Remark 10.2.} Let $F=F(t,x)$ be a solution of the Euler equation (\ref{(Euler-Eq-10.8}), and let $F_0(x)= F(0,x)$ be the initial value. Since $\Delta F(t,x)=(\Delta F_0)(\eta_t^{-1}(x))$, it follows that the vector field $dR_{\eta_t} X_{\Delta F}$ is the restriction of the left-invariant vector field $dL_{\eta} X_{\Delta F_0}$ on the group $\mathcal{D}_{\omega H}$ to the geodesic $\eta_t$.  Therefore, the velocity field $dR_{\eta_t}X_{F}$ along the geodesic $\eta_t$ on $\mathcal{D}_{\omega H}$ assumes a unique value at each point of the geodesic $\eta_t$.  This implies the following conclusion.

\textbf{Conclusion.}   \emph{The geodesics on the group $\mathcal{D}_{\omega H}$ of  the right-invariant metric (\ref{Weak-Riem-D-omega-10.2}) cannot have self-intersections.}

\section{Curvature of the  group $\mathcal{D}_\omega$}\label{Curv-D-omega-10.2}
The  group $\mathcal{D}_{\omega H}$ has the  bi-invariant  weak  Riemannian  structure (\ref{(XF,XH)-D-omega-10.3}).  The curvature of the group $\mathcal{D}_{\omega H}$ with respect to (\ref{(XF,XH)-D-omega-10.3}) is easily found. The covariant derivative $\nabla^0$ of the Riemannian connection of the bi-invariant metric (\ref{(XF,XH)-D-omega-10.3}) on the group $\mathcal{D}_{\omega H}$ has the usual form:
\begin{equation}
\nabla^0_{X_F} X_H = \frac 12 [X_F,X_H]=\frac 12 X_{\{F,H\}}. 
\label{(nabla0-Eq-10.11}
\end{equation}
The curvature tensor of the connection $\nabla^0$ is:
\begin{equation}
R^0(X_F,X_H)X_G =-\frac 14\left [[X_F,X_H],X_G\right ]. 
\label{(R0-Eq-10.12}
\end{equation}

The  sectional  curvature  of  the  group $\mathcal{D}_{\omega H}$ with  respect  to the  bi-invariant  metric  in  the  direction  of  a 2-plane $\sigma$ given by an orthonormal pair of Hamiltonian vector fields $X_F,X_H \in \Gamma_{\omega \partial}(TM)$ is expressed by the formula
\begin{equation}
K_{\sigma} =\frac 14 \int_M \{F,H\}^2\  d\mu. 
\label{(K-sigma-Eq-10.13}
\end{equation}
Thus, the  group $\mathcal{D}_{\omega H}$ is  of  nonnegative  sectional  curvature  and  $K_{\sigma}=0$ iff  the Hamiltonians $F$ and $H$ commute: $\{F,H\}=0$.

Now let us consider the problem on the curvature of the group $\mathcal{D}_\omega$ with respect to the right-invariant weak Riemannian structure (\ref{Weak-Riem-D-omega-10.2}). The corresponding Riemannian connection $\widetilde{\nabla}$ on $\mathcal{D}_\omega$ is defined in the same way as in the case of the group $\mathcal{D}_\mu$. If $X$ and $Y$ are two right-invariant vector fields on $\mathcal{D}_\omega$, then
\begin{equation}
\left(\widetilde{\nabla}_X Y\right)_e = P_e(\nabla_X Y), 
\label{(nabla-Eq-10.14}
\end{equation}
where $P_e:\Gamma(TM)\rightarrow T_e \mathcal{D}_\omega$ is the orthogonal projection of the space $\Gamma(TM)$ of vector fields on $M$ on the space $T_e \mathcal{D}_\omega$ of locally Hamiltonian vector fields and $\nabla$ is the covariant derivative of the metric $g$ on $M$.

\begin{theorem} [see \cite{Smo22}] \label{Th10.5}
The sectional curvature of the group $\sigma$ in the direction of a 2-plane $\sigma$ given by an orthonormal pair of locally Hamiltonian vector fields $X,Y\in T_e \mathcal{D}_\omega$ is expressed by the formula
\begin{multline}
K_\sigma = -\frac 12\left(X,\left[[X,Y],Y\right]\right)_e -\frac
12\left(\left[X,[X,Y]\right],Y\right)_e -\frac 34\left([X,Y],[X,Y]\right)_e - {} \\*
-\left(P_e(\nabla_X X),P_e(\nabla_Y Y)\right)_e + \frac
14\left(P_e(\nabla_X Y+\nabla_Y X),P_e(\nabla_X Y+\nabla_Y X)\right)_e, 
\label{(K-sigma-Eq-10.15}
\end{multline}
where $P_e:\Gamma(TM)\rightarrow \mathcal{D}_\omega$ is the orthogonal projection.
\end{theorem}

For the group $\mathcal{D}_{\omega H}$, we can obtain a more convenient formula for the sectional curvatures expressed through the Hamiltonians of vector fields $X=X_F$ and $Y=X_H$. First, we give the following characterization of the Hamiltonian component $P_e(\nabla_X Y)\in \Gamma_{\omega \partial}(TM)$ of the vector field of the covariant derivative $\nabla_X Y$ on $M$.

\begin{lemma} \label{Lem10.6}
For any Hamiltonian vector fields $X=X_F$ and $Y=X_H$ on $M$,  the  Hamiltonian $S$ of the vector field $X_S=P_e(\nabla_X Y)\in \Gamma_{\omega \partial}(TM)$ is connected with the Hamiltonians $F$  and $H$   by the relation
\begin{equation}
\Delta S =\frac 12\left(\Delta\{F,H\}+\{F,\Delta H\}+\{H,\Delta F\}\right), 
\label{S-Eq-10.16}
\end{equation}
where $\Delta =-\div \circ \grad$ is the Laplacian and $\{F,H\}$ is the Poisson bracket.
\end{lemma}

\begin{proof}
It was demonstrated in \cite{Eb-Mar} that, given the weak right-invariant Riemannian structure (\ref{Weak-Riem-D-omega-10.2}) on $\mathcal{D}_{\omega \partial}$, there exists a Riemannian connection whose covariant derivative $\widetilde{\nabla}$ at the identity $e\in \mathcal{D}_{\omega \partial}$ is given by the formula
$$
(\widetilde{\nabla}_X Y)_e = P_e(\nabla_{X_e} Y_e),
$$
where $\nabla$ is the covariant derivative of the Riemannian connection on $M$ and $X(\eta)=X_e \circ \eta$, $Y(\eta)=Y_e \circ \eta$ are right-invariant vector fields on $\mathcal{D}_{\omega \partial}$, $X_e, Y_e\in T_e\mathcal{D}_{\omega \partial}$.
For determining $P_e(\nabla_{X_e} Y_e) = (\widetilde{\nabla}_X Y)_e$, we use the six-term formula
$$
2(\widetilde{\nabla}_X Y, Z)_e = X(Y,Z)+Y(Z,X)-Z(X,Y)+(Z,[X,Y])_e+(Y,[Z,X])_e-(X,[Y,Z])_e,
$$
where $X=X_F$, $Y=X_H$ and $Z=X_G$ are regarded as right-invariant vector fields on $\mathcal{D}_{\theta}$. Taking the right invariance of the weak Riemannian structure (\ref{Weak-Riem-D-omega-10.2}) into account, we obtain $X(Y,Z)=Y(Z,X)=Z(X,Y)=0$. Using the bi-invariant scalar product (\ref{(XF,XH)-D-omega-10.3}), we obtain
$$
2(\widetilde{\nabla}_X Y, Z)_e =(X_G,[X_F,X_H])_e+(X_H,[X_G,X_F])_e-(X_F,[X_H,X_G])_e =
$$
$$
=(X_G,X_{[F,H]})_e+(X_H,X_{[G,F]})_e-(X_F,X_{[H,G]})_e =
$$
$$
=\langle X_G,X_{\Delta[F,H]}\rangle_e+\langle X_{\Delta H},X_{[G,F]}\rangle_e- \langle X_{\Delta F},X_{[H,G]}\rangle_e =$$
$$
=\langle X_G,X_{\Delta[F,H]}\rangle_e+\langle [X_F,X_{\Delta H}],X_{G}\rangle_e +\langle [X_H,X_{\Delta F}],X_{G}\rangle_e =
$$
$$
=\langle X_{\Delta[F,H]}+X_{[F,\Delta H]} + X_{[H,\Delta F]},X_{G}\rangle_e.
$$
On the other hand, we have
$$
2(\widetilde{\nabla}_X Y, Z)_e =2(P_e(\nabla_{X}Y),Z)_e = 2(X_S,X_G)_e =2\langle X_{\Delta S},X_G\rangle_e.
$$
\end{proof}

\begin{corollary} \label{Cor10.7}
Let $X=X_F$ and $Y=X_H$.  If $P_e(\nabla_X X)=X_S$ and $P_e(\nabla_X Y+\nabla_Y X)=X_T$, then
\begin{equation}
\Delta S =\{F,\Delta F\}, 
\label{Delta-S-Eq-10.17}
\end{equation}
\begin{equation}
\Delta T =\{F,\Delta H\} + \{H,\Delta F\}. 
\label{Delta-T-Eq-10.18}
\end{equation}
\end{corollary}

Using the expressions for $\Delta S$ and $\Delta T$ we obtain from (\ref{(K-sigma-Eq-10.15}) the following expression for $K_\sigma$.

\begin{theorem} \label{Th10.8}
The sectional curvature of the  group $\mathcal{D}_{\omega H}$ (with  metric (\ref{Weak-Riem-D-omega-10.2})) in  the  direction of  a 2-plane $\sigma\in T_e\mathcal{D}_{\omega H}$ given by an orthonormal  pair of  Hamiltonian  vector  fields $X_F,X_H\in \Gamma_{\omega \partial}(TM)$ is expressed by the formula
\begin{multline}
K_\sigma=-\frac 34\int_M \Delta\{F,H\}\{F,H\}\ d\mu + {} \\*
+\frac 12\int_M \{F,H\}(\{F,\Delta H\}+\{\Delta F,H\})\ d\mu -
\int_M \{F,\Delta F\}\Delta^{-1}(\{H,\Delta H\})\ d\mu + {} \\*
+\frac 14 \int_M (\{F,\Delta H\}+\{H,\Delta F\}) \Delta^{-1}(\{F, \Delta H\} +\{H,\Delta F\})\ d\mu.
\label{(K-sigma-Eq-10.19}
\end{multline}
\end{theorem}

In the case where, as the Hamiltonians $F$ and $H$, we take the eigenfunctions of the Laplace operator, $\Delta F = \alpha F$ and $\Delta H = \beta H$, the formula for the sectional curvatures of the group $\mathcal{D}_{\omega H}$  becomes
\begin{multline}
K_\sigma=-\frac 34\int_M \Delta\{F,H\}\{F,H\}\ d\mu + {} \\*
+\frac {\alpha+\beta}{2}\int_M \{F,H\}^2\ d\mu + \frac
{(\alpha-\beta)^2}{4}\int_M \{F,H\}\Delta^{-1} \{F,H\}\ d\mu. 
\label{(K-sigma-Eq-10.20}
\end{multline}

Assume that the structural constants of the Lie algebra $\Gamma_{\omega \partial}(TM)$,
$$
\{F,H\}=C^i_{FH}F_i,
$$
where $F_i$ is the orthonormal system of eigenfunctions of the Laplace operator corresponding to eigenvalues $\lambda_i$, are known. In this case, the formula for the sectional curvature of the group $\mathcal{D}_{\omega H}$ assumes a simpler form
\begin{equation}
K_\sigma=\frac {1}{\alpha \beta}\left( -\frac 34 \sum_{i>0}
\lambda_i (C^i_{FH})^2 +\frac {\alpha+\beta}{2}\sum_{i>0}(C^i_{FH})^2 + \frac
{(\alpha-\beta)^2}{4}\sum_{i>0} \frac{(C^i_{FH})^2}{\lambda_i}\right).
\label{(K-sigma-Eq-10.21}
\end{equation}

In this formula, it is assumed that the functions $F$ and $G$ have  unit $L^2$-norms;  then  $\|X_F\|^2 =\alpha$ and $\|X_H\|^2 =\beta$.  This explains the appearance of the coefficient $\frac {1}{\alpha \beta}$ in formula (\ref{(K-sigma-Eq-10.21}).


\subsubsection*{Structural constants of the Lie algebra $T_e \mathcal{D}_\omega(S^2)$ in the basis of spherical functions}

Introduce the coordinates $(z,\varphi)$ on two-dimensional sphere $S^2$:
$$
\begin{array}{l}
  x = \sqrt{1-z^2}\sin \varphi  \\
  y = \sqrt{1-z^2}\cos \varphi
\end{array},\qquad 0< \varphi < 2\pi,\ -1 < z < 1.
$$
Then the Riemannian volume element $\omega$ on $S^2$ has the form $\omega=dz\wedge d\varphi$. Therefore, $z$ and $\varphi$ are canonical variables on the sphere $S^2$.  Therefore, the Poisson bracket of two functions $F$ and $H$ on $S^2$ is calculated according to the usual formula $\{F,H\}= \frac {\partial H}{\partial z}\frac {\partial F}{\partial \varphi} - \frac {\partial H}{\partial \varphi}\frac {\partial F}{\partial z}$.

The spherical functions
\begin{equation}
Y_m^l(z,\varphi)= \frac {(-1)^l}{2^l l!} \sqrt{\frac
{(2l+1)(l-m)!}{4\pi (l+m)!}}e^{im\varphi}(1-z^2)^{m/2}
\frac{d^{l+m}}{dz^{l+m}}(1-z^2)^{l} , 
\label{Y-ml-Eq-8.6}
\end{equation}
where $l\in \mathbb{N}$ and $-l\leq m \leq l$, compose a complete  orthonormal  function  system on $S^2$ (see, e.g., \cite{La-Li}).  The functions $Y_m^l(z,\varphi)$ are the eigenfunctions of the Laplace operator: $\Delta Y_m^l=l(l+1)Y_m^l$.

The structural constants of the Lie algebra $T_e \mathcal{D}_\omega(S^2)$ in the basis of spherical functions  were  calculated  in  the  work \cite{Ara-Sav} of  Arakelyan  and  Savvidy;  also,  in  this  work,  the  curvature tensor was found.
We set $Y_{lm}(z,\varphi):=Y_m^l(z,\varphi)$; then the structural constants $C_{nmkl}^{ij}$,
$$
\{Y_{nm}, Y_{kl}\} =C_{nm,kl}^{ij}\ Y_{ij},
$$
have the following form \cite{Ara-Sav}:
$$
C_{nm,kl}^{ij} =-i(-1)^j\sqrt{\frac {(2n+1)(2k+1)(2l+1)}{4\pi}}\
l\sum_{p}(2(n-2p-1)+1)\times
$$
$$
\times\sqrt{\frac {(n-|m|)\dots (n-|m|-2p)}{(n+|m|)\dots
(n+|m|-2p)}}\left( \begin{array}{ccc}
  n-2p-1 & k & i \\
  m & l & -j
\end{array}\right)
\left( \begin{array}{ccc}
  n-2p-1 & k & i \\
  0 & 0 & 0
\end{array}\right)-
$$
$$
-m\sum_{q}(2(k-2q-1)+1)\sqrt{\frac {(k-|l|)\dots
(k-|l|-2q)}{(k+|l|)\dots (k+|l|-2q)}}\times
$$
$$
\times \left(
\begin{array}{ccc}
  n & k-2q-1 & i \\
  m & l & -j
\end{array}\right)
\left( \begin{array}{ccc}
  n & k-2q-1 & i \\
  0 & 0 & 0
\end{array}\right),
$$
where $\left( \begin{array}{ccc}
  n & k & i \\
  m & l & j
\end{array}\right)$
are $3j$-Wigner symbols (see, e.g., \cite{La-Li}).

\subsection{Curvature of the symplectic diffeomorphism group of the torus}\label{Curv-D-omega-torus-10.3}
Let $T^{2q}=\mathbb{R}^{2q}/2\pi \mathbb{Z}^{2q}$ be a torus. We assume that the standard symplectic form $\omega =\sum_{i=1}^q dx_i\wedge dy_i$ is given on $T^{2q}$. The functions
$$
\cos(nx+my),\quad  \sin(nx+my),
$$
where $n,m \in \mathbb{Z}^q$, $nx=\sum_{i=1}^q n_i x_i$ compose a complete orthogonal function system on the torus $T^{2q}$.

Consider a small  two-dimensional  area $\sigma\subset \Gamma_{\omega \partial}(TM)$ composed,  e.g.,  by  the  vector  fields $X_F$ and $X_H$ with the Hamiltonians
$$
F =\cos(nx+my),\quad H=\cos(kx+ly),
$$
under the condition $n^2+m^2\neq 0$ and $k^2+l^2 \neq 0$.  The Poisson bracket is
\begin{equation}
\{F,H\} =\frac 12(mk-nl)\left(\cos((n-k)x+(m-l)y)-\cos((n+k)x +(m+l)y)\right). 
\label{({F,H}-Eq-10.22}
\end{equation}
Under the integration over the torus $T =T^{2q}$ with respect to $d\mu =dx_1\wedge ...\wedge dy_q$ we have
$$
\int_T \cos(nx+my)\cos(kx+ly)d\mu =\left\{\begin{array}{l} 0,\quad (n,m)\neq (k,l) \\
  \frac 12 (2\pi)^{2q}, \quad n=k, m=l
\end{array} \right..
$$
Therefore, $\|F\|^2=\frac 12 (2\pi)^{2q}$ and
$$
\int_T \{F,H\}^2\  d\mu = \frac 14(mk-nl)^2\left(\frac 12
(2\pi)^{2q}+\frac 12 (2\pi)^{2q}\right)=(2\pi)^{2q}\frac {(mk-nl)^2}{4}.
$$
For the sectional curvature of the bi-invariant metric
$$
K_\sigma = \frac {1}{4\|F\|^2\|H\|^2} \int_T \{F,H\}^2\  d\mu,
$$
we obtain the expression
\begin{equation}
K_\sigma = \frac {(mk-nl)^2}{4(2\pi)^{2q}}. 
\label{(K_sigma-Eq-10.23}
\end{equation}
In the case $H=\sin(kx+ly)$ and two sines, we obtain exactly the same formula.

Therefore,  the  sectional  curvatures  of  the  bi-invariant  metric  are  nonnegative  and  can  assume  arbitrarily large values.  This yields a geometric explanation of the fact that the group exponential (i.e., the exponential of the bi-invariant metric) does not cover a neighborhood of the identity of the diffeomorphism group.

Let us calculate the sectional curvature of the right-invariant metric on the group $\mathcal{D}_{\omega H}(T^{2q})$.  Consider a two-dimensional small area $\sigma\subset \Gamma_{\omega \partial}(TM)$ composed of the vector fields $X_F$ and $X_H$ with the Hamiltonians
$$
F =\cos(nx+my),\quad H=\cos(kx+ly),
$$
under  the  condition $n^2+m^2\neq 0$ and $k^2+l^2 \neq 0$. These  functions  are  eigenfunctions  of  the  Laplace operator:
$$
\Delta F=(n^2+m^2)F,\ \lambda=n^2+m^2,\qquad \Delta H
=(k^2+l^2)H,\ \mu=k^2+l^2,
$$
Therefore:
$$
\|X_F\|^2 =\lambda \|F\|^2 =\lambda\frac 12 (2\pi)^{2q},\quad
\|X_H\|^2= \mu \|H\|^2 = \mu \frac 12 (2\pi)^{2q},
$$
$$
\int_T \{F,H\}^2\  d\mu = (2\pi)^{2q}\frac {(mk-nl)^2}{4},
$$
$$
\int_T \{F,H\} \Delta\{F,H\}\  d\mu = (2\pi)^{2q}\frac
{(mk-nl)^2}{8}\left((n-k)^2+(m-l)^2+(n+k)^2+(m+l)^2 \right),
$$
$$
\int_T \{F,H\} \Delta^{-1}\{F,H\}\  d\mu = (2\pi)^{2q}\frac
{(mk-nl)^2}{8}\left(\frac {1}{(n-k)^2+(m-l)^2}+\frac
{1}{(n+k)^2+(m+l)^2}\right).
$$
By formula (\ref{(K-sigma-Eq-10.20}), we obtain
\begin{equation}
K_\sigma=-\frac
{(mk-nl)^4(n^2+m^2+k^2+l^2)}{(2\pi)^{2q}(n^2+m^2)(k^2+l^2)((n-k)^2+(m-l)^2)((n+k)^2+
(m+l)^2))}. 
\label{(K_sigma-Eq-10.24}
\end{equation}

We  see  that  the  sectional  curvatures  of  the  right-invariant  metric  are  nonpositive,  and  their  modules grow in a lesser power, which completely corresponds to the fact that the Riemannian exponential of the right-invariant metric is a local diffeomorphism.


%
%

\end{document}